\documentclass[11pt,a4paper]{amsart}
\usepackage{amsfonts}
\usepackage{amsmath}
\usepackage{amssymb, amsbsy}
\usepackage[margin=1.1in]{geometry}
\usepackage{tikz}
\usepackage{amsthm} 
\usepackage[applemac]{inputenc}   % Codage du fichier   
\usepackage{mathrsfs,euscript,anyfontsize}
\usepackage{mathtools}	
\usepackage{datetime}						% uitbreiding van amsmath 
\usepackage{commath}							% Om differentialen te typesetten als operatoren:
\renewcommand{\leq}{\leqslant}
\renewcommand{\geq}{\geqslant}
\DeclareMathOperator{\lcm}{lcm}				% \dif x geeft dx, \od[n]{f}{x} geeft d^{n}f/dx^{n}, ... 
\usepackage{enumitem}  							% hiermee kan je zelf de labels kiezen bij vb. itemize en enumeratie

\usepackage[english]{babel}						% http://ftp.snt.utwente.nl/pub/software/tex/macros/latex/required/babel/base/babel.pdf

\newtheorem{theorem}{Theorem}[section]

\newtheorem{lemma}[theorem]{Lemma}

\theoremstyle{definition}

\numberwithin{equation}{section}

\newcommand{\N}{\mathbb{N}}
\newcommand{\Z}{\mathbb{Z}}

\newcommand{\R}{\mathbb{R}}
\newcommand{\C}{\mathbb{C}}
\newcommand{\e}{{\rm e}}
\newcommand{\ga}{{\mathfrak a}}
\newcommand{\gb}{{\mathfrak b}}
\newcommand{\gc}{{\mathfrak c}}
\newcommand{\gh}{{\mathfrak h}}
\newcommand{\eC}{{\EuScript C}}
\newcommand{\eD}{{\EuScript D}}
\renewcommand{\Re}{\operatorname{Re}}

\makeatletter
\@namedef{subjclassname@2020}{%
  \textup{2020} Mathematics Subject Classification}
\makeatother

\begin{document}

\title{The saddle-point method \goodbreak for general partition functions (*)}
\date{\today, \currenttime}
\author[G.~Debruyne]{Gregory Debruyne}
\thanks{(*) Some corrections with respect to the published version are included here.}
\thanks{G.~Debruyne acknowledges support by Postdoctoral Research Fellowships of the Research Foundation--Flanders and the Belgian American Educational Foundation. The latter one allowed him to do part of this research at the University of Illinois at Urbana-Champaign.} 
\address{Department of Mathematics: Analysis, Logic and Discrete Mathematics\\ \goodbreak Ghent University\\ Krijgslaan 281\\ 9000 Ghent\\ Belgium}
\email{gregory.debruyne@ugent.be}

\author[G.~Tenenbaum]{G\'erald Tenenbaum}
\address{Institut \'Elie Cartan \\ Universit\'e de Lorraine \\ BP 70239 \\ \goodbreak 54506 Vand\oe uvre-l\`es-Nancy Cedex \\ France}
\email{gerald.tenenbaum@univ-lorraine.fr}

\subjclass[2020]{Primary 05A17, 11N37, 11P82; Secondary 11P83.}
\keywords{partitions; saddle-point method; asymptotic formula; abstract partitions}

\begin{abstract}
We apply the saddle-point method to derive asymptotic estimates or asymptotic series for the number of partitions of a natural integer into parts chosen from a subset of the positive integers whose associated Dirichlet series  satisfies certain analytic properties. This enables grouping in a single statement many cases studied in the literature, as well as a number of new ones.
\end{abstract}

\maketitle

%%%%%%%%%%%%%%%%%%%%%%%%%%%%%%%%%%%%%%%%%%%%%%%%%%%%%%%%%%%%%%%%%%%%%%%%%%%%%
% INTRODUCTION
%%%%%%%%%%%%%%%%%%%%%%%%%%%%%%%%%%%%%%%%%%%%%%%%%%%%%%%%%%%%%%%%%%%%%%%%%%%%%
\section{Introduction}

Given a subset $\Lambda$ of the set $\N^*$ of positive integers,\footnote{Throughout this paper, we let $\N$ denote the set of non-negative integers and write $\N^*:=\N\smallsetminus\{0\}$.} we define, for each integer $n\geqslant 1$,  the number $p_\Lambda(n)$  of partitions of $n$ all of whose summands belong to $\Lambda$. In this work, we investigate how the saddle-point method may be employed to derive asymptotic information on $p_\Lambda(n)$ from analytic properties of the associated Dirichlet series 
\begin{equation} L_{\Lambda}(z) := \sum_{m \in \Lambda} m^{-z},
\end{equation} 
initially defined on the half-plane $\Re z > \sigma_c(\Lambda)$, where $\sigma_c(\Lambda)>0$ is the abscissa of convergence.
\par 
Asymptotics for $p_\Lambda(n)$ through the saddle-point method have already been achieved in \cite{Richmond-a, richmond}, which \cite{coons-kirsten} elaborated upon. These papers actually consider moments of partitions, which is more general than the present study. On the other hand, we aim here at full asymptotic expansions rather than just obtaining main terms. Our restrictions on the set $\Lambda$ are also slightly weaker inasmuch we do not require that $1 \in \Lambda$.
\par 
We note right away that, as $\sigma_c(\Lambda)>0$, $\Lambda$ is infinite, and that it follows from the Phragmén--Landau theorem (see, e.g., \cite[th. II.1.9]{GT15}) that $\sigma_c(\Lambda)$ is a singularity of $L_\Lambda$ and so  this series cannot be continued as an entire function.
\par 
For $\Lambda \subseteq \mathbb{N}^*$, we define the \emph{greatest common divisor} of $\Lambda$, and write $\gcd(\Lambda)$, as the greatest natural number $q$ such that $\Lambda \subseteq q\mathbb{N}^* = \{q, 2q, 3q ,\dots \}$.\par 
Given $A\in \R$, we define the {\it class} $\eC(A)$ comprising those subsets  $\Lambda$ fulfilling the following conditions:
\par 
{
\leftskip17mm\rightskip7mm\parindent-7mm
 (a)  $\gcd(\Lambda)=1$;
\par 
 (b) $L_\Lambda$ may be meromorphically continued to the closed half-plane $\Re z \geqslant -\varepsilon$ for suitable $\varepsilon>0$; \par 
(c) this continuation presents a unique simple pole at $z = \sigma_c(\Lambda)$ with residue $A$; 
\par 
(d) we have $|L_\Lambda(-\varepsilon + it)| \ll \e^{a|t|}$ $(t\in\R)$ for some $a < \pi/2$ and $|L_\Lambda(\sigma+ it)| \ll_\delta \exp(\e^{\delta|t|})$ for any $\delta > 0$ where $-\varepsilon  \leq \sigma \leq 2$ and $|t| \geq 1$. \par }
\par
\leftskip=0mm\rightskip=0mm\parindent5mm
We furthermore define the {\it subclass} $\eD(A)$ of $\eC(A)$ comprising those subsets  $\Lambda$ satisfying the extra conditions:
\par 
{\leftskip17mm\rightskip7mm\parindent-7mm
(e) $L_\Lambda$ may be meromorphically continued to $\C$;\par 
(f) for suitable $R_N\to\infty$ and some {$a<\pi/2$}, we have $$|L_\Lambda(-R_{N} + it)| \ll \exp(a|t|)\quad(t\in\R,\,N\to\infty),$$
while $|L_\Lambda(\sigma+ it)| \ll_\delta \exp(\e^{\delta|t|})$ for any $\delta > 0$ on $-R_N \leq \sigma \leq 2$ and $|t| \geq 1$;
\par 

(g) for all $q \geqslant 2$ the set $\Lambda\smallsetminus q\N$ is infinite.
\par }
\begin{theorem} 
\label{thsgpimain} 
Let $A\in\R$ and $\Lambda \in\eC(A)$. Then
\begin{equation} \label{eqsgpirweak}
 p_{\Lambda}(n) \sim \gb\e^{\gc n^{\alpha/(\alpha + 1)}}/n^\gh\qquad (n\to\infty),
\end{equation}
where $\alpha:=\sigma_c(\Lambda)$, $\gh:=(1-L_\Lambda(0) + \alpha/2)/(\alpha + 1)$ and
\begin{align*}
 \ga & := \{A\Gamma(1+\alpha)\zeta(1+\alpha)\}^{1/(\alpha + 1)},\quad
 \gb  := \frac{\e^{L_\Lambda'(0)}\mathfrak{a}^{-L_\Lambda(0) + 1/2}}{\sqrt{2\pi(1+\alpha)}},\quad \gc  := \mathfrak{a}(1+1/\alpha).
\end{align*}
Moreover, if $\Lambda\in\eD(A)$, then there exist constants $\gamma_{j,h}$ $((j,h)\in\N^2)$ such that for each $N\geqslant 1$,
\begin{equation} \label{eqsgpirstrong}
p_{\Lambda}(n) = \frac{\gb\e^{\gc n^{\alpha/(\alpha + 1)}}}{n^\gh} \Bigg\{1 +\sum_{\substack{j+h\geqslant 1\\ j\alpha+h\leqslant N(\alpha+1)}} \frac{\gamma_{j,h}}{ n^{(j\alpha + h)/(\alpha + 1)}} + O\Big(\frac1{n^N}\Big)\Bigg\}.
\end{equation}
\end{theorem}
\par 
Explicit values of the constants $\gamma_{j,h}$ can be obtained from the calculations in our proof. For instance, we have 
\begin{equation}
\label{g10+01}
 \begin{aligned}
 \gamma_{1,0} &= \frac{1}{2(1+\alpha)\ga}\left\{-L_\Lambda(0)^{2} + L_\Lambda(0)(\alpha + 1) + \tfrac14(\alpha + 3)(\alpha + 2) - \tfrac5{12}(\alpha + 2)^{2}\right\},\\
 \gamma_{0,1} &= \tfrac12\ga L_\Lambda(-1).
 \end{aligned}
\end{equation}

\section{The generating function}
For notational simplicity, in the sequel we drop the suffix $\Lambda$ from $p_\Lambda$ and $L_\Lambda$.
With the convention that $p(0)=1$, we set $F$ as the generating function of $p$,
\begin{equation}\label{eqsgpggen}
 F(s) = \sum_{n \in \N} p(n) \e^{-sn} = \prod_{m \in \Lambda} (1-\e^{-sm})^{-1}\qquad (\Re s > 0),
\end{equation}
so that  $p(n)$ may be recovered through Cauchy's formula
\begin{equation}\label{eqsgpgcauchy}
 p(n) = \frac{1}{2\pi i} \int^{\sigma + i\pi}_{\sigma - i\pi } \e^{sn}F(s)\mathrm{d}s = \frac{1}{2\pi} \int^{\pi}_{-\pi } \e^{\sigma n + itn}F(\sigma + it)\mathrm{d}t\qquad (n\geqslant 0,\, \sigma > 0).  
\end{equation}
\par  The principle of the saddle-point method consists in selecting $\sigma$ as the solution of the equation $-F'(\sigma)/F(\sigma) = n$, which we shall denote by $\varrho=\varrho_n$. We shall need some approximations to $\Phi:=\log F$ and its derivatives. We use the determination of $\Phi$ obtained by summing the principal branches of the complex logarithms of each term of the infinite product \eqref{eqsgpggen}. Thus, for $\Re s > 0$, we have
\begin{equation}
 \Phi(s) :=  \log F(s) = -\sum_{m\in \Lambda} \log(1-\e^{-sm}) = \sum_{m \in \Lambda} \sum_{k \geqslant  1} \frac{\e^{-m ks}}{k} = \sum_{n \geqslant  1} \frac{f(n)}{n} \e^{-sn},
\end{equation}
where 
\begin{equation}
 f(n) := \sum_{\substack{m \in \Lambda \\ m \mid n}} m\qquad (n\geqslant 1).
\end{equation}
Note that $-\Phi'(\sigma)$ is strictly decreasing from $\infty$ to $0$ on $(0,\infty)$, so $-\Phi'(\sigma) = n$ has indeed a unique solution. By the Mellin inversion formula for $\e^{-sm}$ and the convolution identity
\begin{equation}
 \sum_{n\geqslant 1} \frac{f(n)}{n^{z+1}} = \zeta(z+1) L(z)\qquad (\Re z > \alpha),
\end{equation}
where $\zeta$ stands for the Riemann zeta function, we obtain
\begin{equation}
 \Phi(s) = \frac{1}{2\pi i} \int^{2+i\infty}_{2-i\infty} \Gamma(z)\zeta(z+1)L(z) \frac{\mathrm{d}z}{s^{z}}\qquad ( \Re s > 0).
\end{equation}
Differentiating under the integral sign, one obtains, for natural $k$,
\begin{equation}
 \Phi^{(k)}(s) = \frac{(-1)^{k}}{2\pi i} \int^{2+i\infty}_{2-i\infty} \Gamma(z+k)\zeta(z+1)L(z) \frac{\mathrm{d}z}{s^{z+k}}\qquad ( \Re s > 0).
\end{equation}
We move the line of integration to $\Re z = -\varepsilon$, resp. (for $\eD(A)$) $\Re z = -R_{N}$. This is allowed since $\zeta$ has finite order in every vertical strip, $\Gamma$ has exponential decay $\ll_{\delta} \e^{-(\pi/2 - \delta)|y|}$ for any $\delta > 0$ on every vertical strip by the complex Stirling formula and $$L(z) - A/(z-\alpha) \ll \e^{a|y|}\quad( -\varepsilon \leqslant x \leqslant 2, {\text{ resp.}} -R_{N} \leqslant x \leqslant 2),$$ by hypothesis and the Phragmén--Lindel\"of principle. (By the same result  we may also assume that $R_{N}$ is not an integer.) Taking the residues into account, we obtain, for $k\geqslant 1$, as $\sigma \rightarrow 0+$, 
\begin{equation}\label{eqsgpgderw}
\begin{aligned} 
 \Phi(\sigma) & = \frac{A\Gamma(\alpha)\zeta(1+\alpha)}{\sigma^\alpha}  -L(0) \log \sigma + L'(0) + O(\sigma^{\varepsilon}),\\
(-1)^{k}\Phi^{(k)}(\sigma) & = \frac{A\Gamma(k+\alpha)\zeta(1+\alpha)}{\sigma^{\alpha+k}}  + \frac{\Gamma(k)L(0)}{\sigma^k} +  O\Big(\frac1{\sigma^{k-\varepsilon}}\Big),
\end{aligned}
\end{equation}
and, in the case $\Lambda\in\eD(A)$,
\begin{equation} \label{eqsgpgders}
\begin{aligned}
\Phi(\sigma) & = \frac{A\Gamma(\alpha)\zeta(1+\alpha)}{\sigma^\alpha}  -L(0) \log \sigma + L'(0) - \zeta(0)L(-1)\sigma\\
& \ \ \  + \sum_{1\leqslant j\leqslant R_{N}/2} \frac{\zeta(1-2j)L(-2j)}{(2j)!}\sigma^{2j} +  O(\sigma^{R_{N}}),\\
(-1)^{k}\Phi^{(k)}(\sigma) & = \frac{A\Gamma(k+\alpha)\zeta(1+\alpha)}{\sigma^{\alpha+k}} + \frac{\Gamma(k)L(0)}{\sigma^k} \\
& \ \ \   + \sum_{1 \leqslant  j < R_{N}+k-1} \frac{(-1)^{j+k-1}\zeta(2-k-j)L(1-k-j)}{(j+k-1)!}\sigma^{j-1} +  O\big(\sigma^{R_{N}-k}\big),
\end{aligned}
\end{equation}
where we have exploited the fact that $\zeta$ vanishes at negative even integers and therefore cancels the corresponding poles of $\Gamma$. We can now approximately solve the saddle-point equation $-\Phi'(\varrho) = n$ via, for example, the method of iteration or the classical approach resting on Rouché's theorem. We find that
\begin{equation}\label{eqsgpginvw}
 \varrho = \frac{\left\{A\Gamma(1+\alpha)\zeta(1+\alpha)\right\}^{1/(1+\alpha)}}{n^{1/(1+\alpha)}} +\frac{L(0)}{(1+\alpha)n} + O\Big(\frac1{n^{1 +\varepsilon/(1+\alpha)}}\Big),
\end{equation}
provided $\varepsilon < \alpha$. In the case $\Lambda\in\eD(A)$, we obtain the existence of constants $c_{j,h}$ such that
\begin{equation} 
\label{eqsgpginvs}
 \varrho = \sum_{(j,h) \in E} \frac{c_{j,h}}{n^{(j\alpha +h)/(1+\alpha)}}+ O\Big(\frac1{n^{1+R_{N}/(1+\alpha)}}\Big),
\end{equation}
where $E$ is, for sufficiently large $Y=Y_N$, the intersection of $[0,Y]^2$ with
\begin{equation}
\big\{(j,1):j\geqslant 0\big\}\cup\big\{(j,2):j\geqslant 1\big\}\cup\big\{(j,2h+u):u=1\text{ or }2,\,j\geqslant u,\,h\geqslant 1\big\}.
\end{equation}
Note that the values of the constants $c_{0,1}$ and $c_{1,1}$ are consistent with formula \eqref{eqsgpginvw}. We also have $$c_{2,1} = \alpha L(0)^{2}/(2(1+\alpha)^{2}\mathfrak{a}),\quad c_{1,2} = \zeta(0)L(-1)\mathfrak{a}/(1+\alpha)$$ thereby enabling us to establish \eqref{g10+01}. 
\par
With the above approximations for $\varrho$, $\Phi$ and its derivatives at hand, we can  proceed to estimate the integral \eqref{eqsgpgcauchy}. The first step consists in bounding the contribution of those $s$ that are sufficiently far from the saddle point $\varrho$.

\section{The contribution away from the saddle point} \label{ssgpc}
In order to deal with the singularities away from $s = \varrho$ in this general setting, we present an argument that differs from that of the corresponding Lemma 2.3 in \cite{t-w-l}. The proof given there\footnote{See the corrected version available on arXiv} relied on the arithmetic structure of the $k$-powers via the use of Weyl's inequality---and actually provides a sharper bound than the one we obtain below. In our general framework we wish to reduce the use of specific arithmetic structure as much as possible. Therefore, we shall only use the necessary arithmetic assumption that $\gcd(\Lambda) = 1$ and the regularity condition \eqref{eqsgpcsize} stated below. Note that \eqref{eqsgpcsize} can readily be deduced from the assumptions of Theorem \ref{thsgpimain} via for example the Wiener-Ikehara Tauberian theorem\footnote{Naturally, the assumptions of Theorem \ref{thsgpimain} provide a stronger error term than \eqref{eqsgpcsize}---see,e.g., \cite[th.II.7.13]{GT15}. The benefit of using \eqref{eqsgpcsize} lies in the fact that our argument may still go through even if the $L$-function admits some additional singularities. The Wiener-Ikehara Tauberian theorem already shows that only a continuous extension of $L(z)$ to $\{\Re z  \geqslant \alpha,\,z\neq\alpha\}$ suffices to conclude \eqref{eqsgpcsize}. Furthermore, even when $L(z) - A/(z-\alpha)$ cannot be continuously extended to the whole line $\Re z =\alpha$, but only to some interval $(\alpha - i\lambda,\alpha + i\lambda)$, one may still appeal to the finite form version of the Wiener-Ikehara theorem due to Graham and Vaaler---see,e.g., \cite[th. III.5.4]{korevaarbook}, to get a regularity condition that is sufficient. It only requires minor modifications to the proof of Lemma \ref{lemsgpcdio}.} with $c = A/\alpha$.\par 
Recalling the definition of $F$ in \eqref{eqsgpggen}, we have the following estimate.
\begin{lemma} \label{lemsgpcdio} Set $1 < \beta < 1 + \alpha/2$. Suppose that $\gcd(\Lambda) = 1$ and that, for a suitable constant $c> 0$, we have
\begin{equation} 
\label{eqsgpcsize}
\big|\Lambda\cap[1,x]\big| \sim cx^{\alpha}\qquad (x\to\infty).
\end{equation}
Then
\begin{equation} 
\label{eqsgpcint}
  \int_{\varrho^{\beta} \leqslant |t| \leqslant \pi } \e^{ itn}F(\varrho + it)\mathrm{d}t \ll \varrho^{2}F(\varrho) ,
\end{equation}
 If additionally $\Lambda\smallsetminus q\mathbb{N}$ is infinite for all $q \geqslant 2$, then the left-hand side of \eqref{eqsgpcint} is $\ll_N \varrho^{N}F(\varrho) $ for any fixed $N$.
\end{lemma}
\begin{proof} Our first goal consists in showing that the contribution of the range $\{t:2\pi\varrho< |t|\leqslant \pi\}$  to the integral in \eqref{eqsgpcint} is compatible with the required estimate. By symmetry, we may restrict to the interval $J:=[2\pi\varrho,\pi]$. To achieve this first objective, we shall prove that, under the stated assumptions, we have $|F(\varrho + it)| \ll \varrho^{2}F(\varrho)$ except for a set of measure $\ll \varrho$ on which we have at least $|F(\varrho + it)| \ll\varrho F(\varrho)$. \par 
We note right away that hypothesis \eqref{eqsgpcsize} implies, for suitably large $R_0$,
\begin{equation}
\label{dyadLam}
\big|\Lambda\cap]R,2R]\big|\gg R^\alpha\qquad (R>R_0).
\end{equation}
\par
By Dirichlet's approximation theorem every $t\in J$ may be represented in the form $t/2\pi=a/q  \pm  r$ with $(a,q) = 1$, $q \leqslant 3R_0$ and $0\leqslant r \leqslant 1/3qR_0$.\par 
We start by analyzing the contribution of the ``minor arcs" comprising those $t\in J$ such that $2\varrho/3q < r \leqslant 1/3qR_0$. We note that we do allow the fraction $a = 0$, $q = 1$ for the minor arcs, but not for the major arcs later on. Since
\begin{equation*}
 |1-\e^{-m(\varrho+it)}|^{2} = (1-\e^{-m\varrho})^{2} + 2\e^{-m\varrho}\{1-\cos(mt)\},
\end{equation*}
we have, defining classically $\|\vartheta\|$ as the distance from the real number $\vartheta$ to the set of integers,
\begin{align*}
 \frac{|F(\varrho + it)|}{F(\varrho)} & = \prod_{m \in \Lambda} \bigg(1+ \frac{4\sin^{2}(mt/2)}{\e^{m\varrho}(1-\e^{-m\varrho})^{2}}\bigg)^{-1/2} \leqslant \prod_{m \in \Lambda} \bigg(1+ \frac{16\left\| mt/2\pi\right\|^{2}}{\e^{m\varrho}(1-\e^{-m\varrho})^{2}}\bigg)^{-1/2}\\
& \leqslant \prod_{\substack{m \leqslant 1/\varrho\\ m \in \Lambda }} \bigg(1+ \frac{5\left\| mt/2\pi\right\|^{2}}{(1-\e^{-m\varrho})^{2}}\bigg)^{-1/2} \leqslant \prod_{\substack{1/3rq \leqslant m \leqslant 2/3rq \\ m \in \Lambda}} \bigg(1+ \frac{5\left\| mt/2\pi\right\|^{2}}{ m^{2}\varrho^{2}}\bigg)^{-1/2}.
\end{align*}
Now, for $1/3rq\leqslant m\leqslant 2/3rq$, we have $1/3q\leqslant |mt/2\pi-ma/q|\leqslant 2/3q$ and so $\|mt/2\pi\|\geqslant 1/3q$. Thus
\begin{equation*}
 \frac{|F(\varrho + it)|}{F(\varrho)} \leqslant \prod_{\substack{1/3rq \leqslant m \leqslant 2/3rq \\ m \in \Lambda}} \bigg(1+ \frac{5}{9 q^{2}m^{2}\varrho^{2}}\bigg)^{-1/2} \leqslant \prod_{\substack{1/3rq \leqslant m \leqslant 2/3rq \\ m \in \Lambda}} \bigg(1+ \frac{5r^{2}}{4\varrho^{2}}\bigg)^{-1/2}.
\end{equation*}
If $2\varrho/3q \leqslant r \leqslant \sqrt{\varrho}$, we get
$$\frac{|F(\varrho + it)|}{F(\varrho)}\leqslant \Big(1+\frac{5}{9q^2}\Big)^{-c_0/(rq)^\alpha}\ll \e^{-c_1/\varrho^{\alpha/2}}$$
where $c_0$ and $c_1$ only depend on $R_0$.\par 
When $\sqrt{\varrho}< r \leqslant 1/3qR_0$, we deduce from \eqref{dyadLam} that $T:=\big|\Lambda\cap[1/3rq,2/3rq]\big|\geqslant 4$ for fixed $R_0$ and sufficiently large $n$. It follows that
\begin{equation} 
\label{eqsgpcmin2}
 \frac{|F(\varrho + it)|}{F(\varrho)} \leqslant \Big(1+ \frac{5}{4\varrho}\Big)^{-T/2}\leqslant \varrho^{2}.
\end{equation}
\par
This completes the treatment of the so-called minor arcs.\par 
We next consider the contribution of the ``major arcs'', corresponding to the case $$0\leqslant r \leqslant 2\varrho/3q.$$ Let $m_{q}:=\min\big\{\Lambda\smallsetminus q\N\big\}$ ($2 \leqslant q \leqslant 3R_0$). Observe that $\varrho \leqslant 1/m_{q}$ for large $n$, and so $\|m_{q}t/2\pi\| \geqslant 1/3q$ for all $t\in J$ belonging to the major arcs (whose Dirichlet approximation is a fraction with denominator $q$), hence bounding the products as before yields 
\begin{equation} 
\label{eqsgpcmaj}
 \frac{|F(\varrho + it)|}{F(\varrho)} \leqslant \bigg( 1 + \frac{5\|m_{q}t/2\pi\|^{2}}{m_{q}^{2}\varrho^{2}}\bigg)^{-1/2}\leqslant  \bigg( 1 + \frac{5}{9 q^{2}m_{q}^{2}\varrho^{2}}\bigg)^{-1/2} \leqslant C_{q} \varrho.
\end{equation}
However the overall measure of the major arcs is 
$$\sum_{2\leqslant q\leqslant 3R_0}2\varphi(q)\varrho/3q \ll \varrho,$$
where the implied constant only depends on $R_0$.
Therefore we infer from \eqref{eqsgpcmaj} that the contribution of the major arcs to the integral \eqref{eqsgpcint} is $\ll\varrho^2F(\varrho)$, thereby completing the proof of our first goal.
\par
Hardly anything changes when $\Lambda\smallsetminus q\mathbb{N}$ is infinite for all natural $q \geqslant 2$. Indeed, on the one hand, we may take $T$ as large as we wish in  \eqref{eqsgpcmin2} and, on the other hand, this extra assumption implies that, instead of a single element $m_q$, we may select $N$ ones from $\Lambda\smallsetminus q\N$ when evaluating the contribution of the major arcs, so that the  upper bound corresponding to  \eqref{eqsgpcmaj} becomes $\ll_N\varrho^N$. 
\par 
It remains to bound the contribution to \eqref{eqsgpcint} coming from the range $\varrho^{\beta} \leqslant t\leqslant 2\pi \varrho$. To this end, we can employ a similar but substantially simpler argument. Indeed, we now have $\|mt/2\pi\|=mt/2\pi$ for $m\leqslant 1/2\varrho$. Thus, with $T:=\big|\Lambda\cap[1,1/2\varrho]\big|\gg 1/\varrho^\alpha$ in view of \eqref{eqsgpcsize}, 
\begin{align*}
 \frac{|F(\varrho + it)|}{F(\varrho)} & \leqslant \prod_{\substack{m \leqslant 1/2\varrho \\ m \in \Lambda}} \bigg(1 + \frac{9\big\|mt/2\pi\big\|^{2}}{m^{2}\varrho^{2}}\bigg)^{-1/2} \leqslant  \Big(1 + \tfrac15\varrho^{2\beta - 2}\Big)^{-T/2}\\&\ll\e^{-c_2\varrho^{2\beta-2-\alpha}} \ll_N \varrho^N.
\end{align*}
This completes the proof of \eqref{eqsgpcint}.
\end{proof}
\break
\section{The contribution of the saddle point:\goodbreak completion of the proof of Theorem \ref{thsgpimain}}
By the results of Section \ref{ssgpc} and \eqref{eqsgpgcauchy}, it only remains to estimate 
\begin{equation*}
 I := \int_{|t| \leqslant \varrho^{\beta}} \e^{\Phi(\varrho + it) + itn}\mathrm{d}t,
\end{equation*}
for some $\beta < 1 + \alpha /2$. We shall expand $\Phi(\varrho + it)$ as a Taylor series at $\varrho$. The estimation of the error terms requires that $\beta > 1 + \alpha/3$, so that $|t^{k}\Phi^{(k)}(\varrho)| \ll 1$ for $k \geqslant 3$ and $\ll \varrho^{N}$ for $k \geqslant 3(N+\alpha)/\alpha$, as implied by \eqref{eqsgpgderw}. 
Write 
$$E(t):=\sum_{ 3\leqslant k\leqslant3(N + \alpha)/\alpha}\frac{(it)^{k} \Phi^{(k)}(\varrho)}{k!}.$$
Since $\Phi'(\varrho)+n=0$ by construction,  we obtain, for $H>N/(3\beta-3-\alpha)$,
\begin{align*}
 I & = F(\varrho) \int_{|t| \leqslant \varrho^{\beta}}\e^{-t^2\Phi''(\varrho)/2+E(t) + O(\varrho^{N})}\mathrm{d}t\\
& = F(\varrho) \int_{|t| \leqslant \varrho^{\beta}}e^{-t^2\Phi''(\varrho)/2}\bigg\{1 + \sum_{1\leqslant h\leqslant H}\frac{E(t)^{h}}{h!} + O\big(\varrho^{N}\big)\bigg\}\mathrm{d}t\\
& = F(\varrho) \int_{|t| \leqslant \varrho^{\beta}} e^{-t^2\Phi''(\varrho)/2} \bigg\{1 + \sum_{3\leqslant k\leqslant 3H(N+\alpha)/\alpha}\lambda_{k}(\varrho)t^{k} + O\big(\varrho^{N})\bigg\}\mathrm{d}t,
\end{align*}
where 
\begin{equation} 
\label{eqsgpslambda}
  \lambda_{k}(\varrho) = i^{k} \sum_{1\leqslant h\leqslant H} \frac{1}{h!} \sum_{\substack{3 \leqslant m_{1},\dots,m_{h} \leqslant 3(N + \alpha)/\alpha \\ m_{1} + \dots + m_{h} = m}} \prod_{1 \leqslant j \leqslant h} \frac{\Phi^{(m_{j})}(\varrho)}{m_{j}!}.
\end{equation}
Since the integration range is symmetrical around the origin, only even $k$ contribute. As
\begin{equation*}
 \int_{|t| \leqslant \varrho^{\beta}} e^{-t^2\Phi''(\varrho)/2}t^{2k}\mathrm{d}t = \frac{\sqrt{2\pi}(2k)!}{2^{k}k! \Phi''(\varrho)^{k+1/2}} + O\Big(\e^{-c\varrho^{2\beta}/\Phi''(\varrho)}\Big), 
\end{equation*}
and the above error term is admissible in view of the assumption to $\beta < 1+\alpha/2$, we find
\begin{equation} \label{eqsgpsmaini}
 I = \sqrt{2\pi}F(\varrho)\Bigg\{\frac1{\sqrt{\Phi''(\varrho)}} + \sum_{2\leqslant k\leqslant 3H(N + \alpha)/2\alpha} \frac{(2k)!\lambda_{2k}(\varrho)}{2^{k}k! \Phi''(\varrho)^{k+1/2}} + O\big(\varrho^{N}\big)\Bigg\}.
\end{equation}
It is easily checked that the terms in the sum are of lower order than the main term. Expressing $\varrho$, $\lambda_{2k}(\varrho)$, and $\Phi^{(j)}(\varrho)$  in terms of $n$ via  formulas \eqref{eqsgpslambda}, \eqref{eqsgpgderw}---resp. \eqref{eqsgpgders}---, and \eqref{eqsgpginvw}---resp. \eqref{eqsgpginvs}---, carrying \eqref{eqsgpsmaini} back into \eqref{eqsgpgcauchy} and employing Lemma \ref{lemsgpcdio}, concludes the proof of \eqref{eqsgpirweak}---resp. \eqref{eqsgpirstrong}.

\section{Examples}

In this final section, we illustrate the  applicability of Theorem \ref{thsgpimain} by describing several examples, presented in increasing complexity and generality. Example \ref{expol} generalizes all  previous ones.

\subsection{The classical partitions} In this case $\Lambda = \N^*$. Then $L(z) = \zeta(z)$ and all hypotheses of Theorem \ref{thsgpimain} are clearly satisfied. Since $\zeta(2) = \pi^{2}/6$, $\zeta(0) = -1/2$, $\zeta'(0) = -\log(2\pi)/2$, and $\zeta(-1) = -1/12$, we obtain, with the usual convention of the theory of asymptotic series,
\begin{equation}
 p(n) \sim \frac{\e^{\pi\sqrt{2n/3}}}{4n\sqrt{3}}\bigg\{1 + \sum_{k\geqslant 1} \frac{c_k}{n^{k/2}}\bigg\},
\end{equation}
where $c_{1} = \gamma_{1,0} + \gamma_{0,1} = -\sqrt{2/3}\big(\pi/48 + 3/2\pi\big)$.

\subsection{$\pmb{k}$-th powers} Let $k$ be a natural number and let $\Lambda = \{n^{k}: n \in \N^*\}$. The problem of determining asymptotics for partitions in $k$-th powers has a long history. In 1918, Hardy and Ramanujan \cite{h-r} provided the corresponding formula \eqref{eqsgpirweak} without proof. In 1934, introducing several complicated objects, Wright \cite{wright} obtained an asymptotic expansion for the relevant $p(n)$, noted as $p_k(n)$. More recently, appealing to the Hardy-Littlewood circle method, Vaughan \cite{vaughan} obtained an asymptotic expansion when $k = 2$ and Gafni \cite{gafni} generalized his argument to arbitrary, fixed $k$. Finally, in collaboration with Wu and Li, Tenenbaum obtained \eqref{eqsgpekth} with the saddle-point method \cite{t-w-l}. We refer to the introduction of \cite{t-w-l} for a more detailed account on the history of this problem.\par
As $L(z) = \zeta(kz)$, all hypotheses of Theorem \ref{thsgpimain} are fulfilled with $A = \alpha = 1/k$. We find
\begin{equation} \label{eqsgpekth}
 p_{k}(n) \sim \frac{\mathfrak{b}_{k}\e^{\mathfrak{c}_{k} n^{1/(k + 1)}}}{n^{(3k+1)/(2k+2)}} \bigg\{1 +  \sum_{h\geqslant 1} \frac{c_{k,h}}{ n^{h/(k + 1)}}\bigg\},
\end{equation} 
where
\begin{align*}
 \mathfrak{a}_{k} & = (k^{-1}\Gamma(1+k^{-1})\zeta(1+k^{-1}))^{k/(k + 1)},\\
 \mathfrak{b}_{k} & =\frac{\mathfrak{a}_{k}}{\sqrt{(2\pi)^{k+1}(1+1/k)}},\\
 \mathfrak{c}_{k} & = (k+1)\mathfrak{a}_{k}.
\end{align*}
Furthermore, when $k \geqslant 2$, then $c_{k,1} = \gamma_{1,0} = -(11k^{2} + 11k + 2)/24k\mathfrak{c}_{k}$.

%\subsection{Arithmetic progressions} Here we set $\Lambda = \{qn + a : n \in \mathbb{N}\}$ where $(a,q) = 1$ and $a < q$ and let $p_{a,q}(n)$ denote the corresponding partition function. In this case $L(z)$ becomes $$\sum_{n \geqslant  0} (qn + a)^{-z} = q^{-z}\zeta(z,a/q),$$ where $\zeta(z,a/q)$ denotes the Hurwitz zeta-function. It is a classical fact\footnote{See, e.g., \cite[ex. 186]{GT15}.}  that $\zeta(z,t)$ admits for all $t\in(0,1)$ an analytic extension (with the required bounds to apply Theorem \ref{thsgpimain}) to the whole complex plane except for a simple pole at $1$ with residue $1$. The following specific values for the Hurwitz zeta function are also well-known:
% $$\zeta(0,a/q) = \tfrac12 - \frac aq, \ \zeta'(0,a/q) = -\tfrac12\log(2\pi)  + \log \Gamma(a/q), \ \zeta(-1,a/q) = -\tfrac1{12} +\frac a{2q} - \frac{a^2}{2q^2}\cdot$$ Finally,  since $(a,q) = 1$, the arithmetic conditions (a) and (g) are also fulfilled. Therefore, we may state that 
%\begin{equation}
% p_{a,q}(n) \sim \frac{\Gamma\left(a/q\right) \pi^{a/q - 1} q^{a/2q - 1/2}\e^{\pi\sqrt{2n/3q}}}{2^{3/2 + a/2q} \cdot 3 ^{a/2q} \cdot n^{a/2q + 1/2}}\bigg\{1 + \sum_{h\geqslant 1} \frac{c_{a,q,h}}{n^{h/2}} \bigg\},
%\end{equation}
%where $c_{a,q,1} = \gamma_{1,0} + \gamma_{0,1} = -\sqrt{3q/2}\big\{a(1+a/q)/2\pi q + \pi(1/24 -a/4q + a^{2}/4q^{2})/3q\big\}$.

\subsection{$\pmb{k}$-th powers of arithmetic progressions} We set here $\Lambda := \{(qn + a)^{k} : n \in \mathbb{N}\}$, where $k \geqslant 1$ is natural and $(a,q) = 1$. The problem of finding asymptotics for its partitions, say $p_{a,q,k}(n)$, appears to have been first studied by Berndt, Malik and Zaharescu \cite{b-m-z} by means of the circle method. We have $$L(z) = \sum_{n \geqslant  0} (qn + a)^{-kz} = q^{-kz}\zeta(kz,a/q),$$ 
where $\zeta(kz,a/q)$ denotes the Hurwitz zeta function. It is a classical fact\footnote{See, e.g., \cite[ex. 186]{GT15}.}  that $\zeta(z,t)$ admits for all $t\in(0,1)$ an analytic extension (with the required bounds to apply Theorem \ref{thsgpimain}) to the whole complex plane except for a simple pole at $1$ with residue $1$. The following specific values for the Hurwitz zeta function are also well-known:
 $$\zeta(0,a/q) = \tfrac12 - \frac aq, \ \zeta'(0,a/q) = -\tfrac12\log(2\pi)  + \log \Gamma(a/q), \ \zeta(-1,a/q) = -\tfrac1{12} +\frac a{2q} - \frac{a^2}{2q^2}\cdot$$ Finally,  since $(a,q) = 1$, the arithmetic conditions (a) and (g) are also fulfilled. Therefore, $\Lambda\in\eD(1/qk)$ and we may state that
 %This series has a meromorphic continuation to the whole complex plane, with a single, simple pole at $1/k$ having residue $1/qk$ and $\Lambda\in\eD(1/qk)$. Theorem \ref{thsgpimain} yields
\begin{equation}
 p_{a,q,k}(n) \sim  \frac{\mathfrak{b}_{a,q,k}\e^{\mathfrak{c}_{a,q,k}n^{1/(k+1)}}}{n^{(qk+2ak+q)/(2q(k+1))}}\bigg\{1 + \sum_{h\geqslant 1} \frac{c_{a,q,k,h}} {n^{h/(k+1)}}\bigg\},
\end{equation}
where
\begin{align*}
 \mathfrak{a}_{a,q,k} & = \Big\{k^{-1}q^{-1}\Gamma(1+k^{-1})\zeta(1+k^{-1})\Big\}^{k/(k + 1)},\\
 \mathfrak{b}_{a,q,k} & =\frac{\mathfrak{a}_{a,q,k}^{a/q}\Gamma(a/q)^{k}q^{ak/q}}{\sqrt{(2\pi)^{k+1}q^{k}(1+1/k)}},\quad
 \mathfrak{c}_{a,q,k}  = (k+1)\mathfrak{a}_{a,q,k}.
\end{align*}
Furthermore, when $k \geqslant 2$, then $$c_{a,q,k,1} = \gamma_{1,0} = -\frac{1}{24k\mathfrak{c}_{a,q,k}}\left\{\left(\frac{12a^{2}}{q^{2}} - 1\right)k^{2} + \left(\frac{12a}{q} - 1\right)k + 2\right\}.$$ When $k = 1$, we obtain 
\begin{equation}
 p_{a,q,1}(n) \sim \frac{\Gamma\left(a/q\right) \pi^{a/q - 1} q^{a/2q - 1/2}\e^{\pi\sqrt{2n/3q}}}{2^{3/2 + a/2q} \cdot 3 ^{a/2q} \cdot n^{a/2q + 1/2}}\bigg\{1 + \sum_{h\geqslant 1} \frac{c_{a,q,1,h}}{n^{h/2}} \bigg\},
\end{equation}
where $$c_{a,q,1,1} = \gamma_{1,0} + \gamma_{0,1} = -\sqrt{\frac{3q}{2}}\left\{\frac{a}{2\pi q}\left(1+\frac{a}{q}\right) + \frac{\pi}{3q}\left(\tfrac{1}{24} - \frac{a}{4q} + \frac{a^{2}}{4q^{2}}\right)\right\}.$$

\subsection{Polynomials} \label{expol} The previous examples may be generalized even further. Let $f$ be a polynomial with integer coefficients\footnote{In order to maintain the analogy with arithmetic progressions, we let $f(0)$ also represent a part. This also simplifies the notation in our results. The polynomials representing the previous examples are then $n + 1$, $(n+1)^{k}$, $qn + a$ and $(qn + a)^k$ respectively.} such that $f(\mathbb{N}) \subseteq \mathbb{N}^*$. For convenience, we impose the restriction that $f$ is injective on $\mathbb{N}$: this avoids different $n$ representing the same part\footnote{Even for  non-injective polynomials, our formulas would be valid if one considers parts coming from different $n$ to be different. These would then be colored partitions. For example, for $f(n) = n(n-2) + 2$, the part $2$ arises with multiplicity $2$ due to $f(0) = f(2) = 2$.}. The question of obtaining asymptotics for the partition function, say $p_f(n)$, associated to $\Lambda = \{f(n):n \in \mathbb{N}\}$ when $f$ has non-negative coefficients was raised at the end of the paper \cite{b-m-z}. Elaborating on ideas from the circle method, Dunn and Robles \cite{d-r} established an asymptotic formula for $p_{f}(n)$ under certain fairly restrictive hypotheses on $f$. With our notation, they  require, for example when $f$ is a polynomial of degree $k$, that the coefficient of $n^{k-1}$ in the polynomial $f(n-1)$ should vanish. This has as consequence that their results are not applicable for $k$-th powers of arithmetic progressions unless $a = q = 1$. They also did not explicitly compute the constants $\mathfrak{b}_{f}$ and $\mathfrak{c}_{f}$ of the formula \eqref{eqsgpep} below, but mentioned they are effectively computable. The proof presented here, based on the saddle-point method, appears substantially simpler and provides more general results.\par
Let $f(n) = a_{0}n^{k} + a_{1}n^{k-1} + \dots + a_{k}$. We first analyze the corresponding function $L$. Let $M$ be  sufficiently large. For $\operatorname*{Re} z > 1/k$, we have
\begin{equation}
\label{Lz}
\begin{aligned}
 L(z)  &= \sum_{n \geqslant  0} f(n)^{-z} = a_{k}^{-z} + \sum_{1\leqslant n\leqslant M} f(n)^{-z} +  a_{0}^{-z}\sum_{n > M} n^{-kz}\big\{1+G(n)\big\}^{-z}\\
 &=a_{k}^{-z} + \sum_{1\leqslant n\leqslant M} f(n)^{-z} +  a_{0}^{-z}\sum_{n>M}n^{-kz}\sum_{h\geqslant 0}\binom{-z}hG(n)^h,
 \end{aligned}
 \end{equation}
 where $G(n):=\sum_{1\leqslant j\leqslant k}a_j/(a_0n^j)$. Expanding $G(n)^h$ in the inner sum by the multinomial formula and inverting summations we see that the last double sum may be written as an infinite, convergent, linear combination of terms $\zeta(kz+j)$ where $j$ is integer.\par 
 This provides a meromorphic extension of $L$ to the complex plane. We observe that $L$ has a simple pole at $z = 1/k$ with residue $A:=1/ka_{0}^{1/k}$ and potentially admits simple poles at $z = -r/k$ for natural $r$, but not at negative integers. It is also readily seen that $L$ has at most polynomial growth on any right half-plane from which a neighborhood of the poles has been removed. Thus $\Lambda\in\eC(A)$ and \eqref{eqsgpirweak} holds. However $\Lambda$ does not necessarily belong to $\eD(A)$ and we cannot get \eqref{eqsgpirstrong} immediately. \par 
Regarding the analytical nature of $L(z)$, we observe that the possible presence of additional simple poles does not conceptually alter the strategy employed to prove Theorem \ref{thsgpimain}, and the argument given in previous sections may be adapted to obtain the required asymptotic expansion. The only difference is that these poles give rise to additional terms in \eqref{eqsgpgders} that have to be taken into account. Note that the possible poles of $L$ do not interfere with those of $\Gamma$ at integers. Calculations show that these potential extra poles do not alter the expansion \eqref{eqsgpirstrong}, although they might affect the values $\gamma_{j,h}$.
 \par
We conclude the analysis of the function $L$ by calculating $L(0)$ and $L'(0)$. Employing the formula for the analytic continuation of $L$ above, we find
\begin{equation*}
 L(0) = 1 + M + \zeta(0) - M - \lim_{z \rightarrow 0} \frac{a_{1}}{a_0}z \zeta(kz+1) = \tfrac{1}{2} - \frac{a_{1}}{a_{0}k}.
\end{equation*}
$L'(0)$ can be computed in a similar but slightly more complicated way. Starting from the expression $L'(z) = -\sum_{n\geqslant  0} \log(f(n))/f(n)^z$, the meromorphic extension of $L'(z)$ may be defined by a formula analogous to \eqref{Lz} above. Specializing to $z = 0$ and inserting when needed the formula
$$\sum_{n\geqslant 1} \Big\{\log\Big(1-\frac\alpha n\Big) + \frac\alpha n\Big\} = -\log \Gamma(-\alpha) +\gamma \alpha - \log(-\alpha),$$ where $\gamma$ is the Euler-Mascheroni constant and $\alpha \in \mathbb{C}\smallsetminus (-\mathbb{N}),$ we obtain\footnote{Strictly speaking $\mathrm{mod} \: 2\pi i$, since selecting any particular branch of the logarithm turns out to be irrelevant.}
\begin{equation*}
 L'(0) = -\left(\frac{1}{2} - \frac{a_{1}}{a_{0}k}\right) \log a_{0} - \tfrac12k \log(2\pi) + \sum_{1\leqslant j\leqslant k} \log \Gamma(-\alpha_{j}),
\end{equation*}
where $\alpha_{j}$ are the zeros of $f$. Observe that $\alpha_{j}$ cannot be natural or zero: otherwise $f(\alpha_{j})$ would furnish the part $0$, in contradiction with the assumption $f(\mathbb{N}) \subseteq \mathbb{N}^*$.\par 
It remains to address the question of the validity of condition (g) in the definition of $\eD(A)$. It translates to the existence, for each prime $p$, of integers $n$  such that $f(n) \not\equiv 0$ $(\mathrm{mod} \: p)$. In other words, for each prime factor $p$ of $a_{k}$, the polynomial $f$ reduced $\mathrm{mod} \: p$ should  not divide $n^{p-1} - 1$ (in $\mathbb{F}_{p}$). Note that $a_{k} = f(0) \neq 0$. This restriction excludes for example  the polynomial $n^{2} + n + 2$ corresponding to partitions all parts of which are even. For this polynomial however, one may still use our analysis to find the asymptotics for even $n$, that is $p_{f}(2n) = p_{f/2}(n)$. The arithmetic condition for $f/2$ would then be fulfilled: extending the analysis of $L$ to a polynomial with coefficients in $\Z/2$ does not introduce significant difficulties.
\par
In conclusion, for each $f \in \mathbb{Z}[x]$, for which $f(\mathbb{N}) \subseteq \mathbb{N}^*$, $f$ is injective on $\mathbb{N}$ and such that $f$ does not vanish identically $\mathrm{mod} \: p$ for any prime $p$,  we have 
\begin{equation} \label{eqsgpep}
 p_{f}(n) \sim \frac{\mathfrak{b_{f}} \e^{\mathfrak{c_{f}}n^{1/(k+1)}}}{n^{(k + 1 + 2a_{1}/a_{0})/(2k+2)}}\bigg\{ 1 + \sum_{h\geqslant 1} \frac{c_{f,h}}{ n^{h/(k+1)}}\bigg\},
\end{equation}
where 
\begin{align*}
 \mathfrak{a}_{f} & = \big\{k^{-1}a_{0}^{-1/k}\Gamma(1+k^{-1})\zeta(1+k^{-1})\big\}^{k/(k + 1)},\\
 \mathfrak{b}_{f} & =\frac{\mathfrak{a}_{f}^{a_{1}/a_{0}k}  a_{0}^{-1/2 + a_{1}/a_{0}k} \prod_{j = 1}^{k}\Gamma(-\alpha_{j})}{(2\pi)^{(k+1)/2}\sqrt{1+1/k}},\quad
 \mathfrak{c}_{f}  = (k+1)\mathfrak{a}_{f},
\end{align*}
for a polynomial $f$ of degree $k$, where $a_{0}$ is the coefficient of the dominant term, $a_{1}$ that of $n^{k-1}$ and the $\alpha_{j}$ are the zeros of $f$.
\subsection{Combinations} Theorem \ref{thsgpimain} may also be applied for combinations of the previous instances. For example, let $\Lambda := (a+q\N) \cup (b+r\N)$ and  define $c$ by $(a+q\N) \cap (b+r\N)=c+d\N$ with $d:=\lcm(q,r)$. If this intersection is empty, e.g. when $q=r$ and $a \neq b$, then the function representing the intersection may be omitted in the formula for $L_\Lambda$. Then $L_\Lambda(z)$  becomes $L_{a,q}(z) + L_{b,r}(z) - L_{c,d}(z)$, where $L_{u,v}$ denotes the  $L$-function  corresponding to the arithmetic progression $u+v\N$. The new $L$-function inherits the relevant properties from the old ones. Condition (a) becomes $(a,b,q,r) = 1$. Thus Theorem \ref{thsgpimain} is applicable.\par
Another example is provided by $\Lambda := k\mathbb{N}^{*} \cup \{a\}$, with $k \geqslant 2$. Then condition (a) is equivalent to $(a,k) = 1$, while condition (g) is not fulfilled. Letting $p_{k,a}(n)$ denote the present partition function, we have 
\begin{equation}
 p_{k,a}(n) \sim \frac{\sqrt{k} \e^{\pi \sqrt{2n/3k}}}{2a \pi \sqrt{2n}}\cdot 
\end{equation} 
One may also devise combinations that result in an $L$-function having several poles on the positive real axis. For example, when $\Lambda$ consists of all $k$ and $h$-th powers, the resulting $L$-function can have poles at $1/k$, $1/h$ and $1/\lcm(k,h)$. Although a direct application of Theorem \ref{thsgpimain} is no longer possible, one may still use its proof. As previously mentioned, these poles then have implications for the function $\Phi$ in \eqref{eqsgpgderw} and \eqref{eqsgpgders} and one has to take the extra terms into account in the calculations for the rest of the proof. As the poles now lie at the right of $0$, the extra poles will now also influence the main term.
%%%%%%%%%%%%%%%%%%%%%%%%%%%%%%%%%%%%%%%%%%%%%%%%%%%%%%%%%%%%%%%%%%%%%%%%%%%%%
% BIBLIOGRAPHY
%%%%%%%%%%%%%%%%%%%%%%%%%%%%%%%%%%%%%%%%%%%%%%%%%%%%%%%%%%%%%%%%%%%%%%%%%%%%%

\end{document}